\documentclass[10pt]{amsart}
\usepackage{amssymb,MnSymbol,times}
\usepackage{amsthm,amsmath,cite}

\title[On the best constants associated with $n$-distances]{On the best constants associated with $n$-distances}\thanks{Corresponding author: Jean-Luc Marichal is with the Mathematics Research Unit, University of Luxembourg, Maison du Nombre, 6, avenue de la Fonte, L-4364 Esch-sur-Alzette, Luxembourg. Email: jean-luc.marichal[at]uni.lu}

\author{Gergely Kiss}
\address{Alfr\'ed R\'enyi Institute of Mathematics, Re\'altanoda u.\ 13-15, HU-1053 Budapest, Hungary}
\email{kigergo57[at]gmail.com}

\author{Jean-Luc Marichal}
\address{Mathematics Research Unit, University of Luxembourg, Maison du Nombre, 6, avenue de la Fonte, L-4364 Esch-sur-Alzette, Luxembourg}
\email{jean-luc.marichal[at]uni.lu}

\date{November 27, 2019}

\begin{document}

\theoremstyle{plain}
\newtheorem{theorem}{Theorem}[section]
\newtheorem{lemma}[theorem]{Lemma}
\newtheorem{proposition}[theorem]{Proposition}
\newtheorem{corollary}[theorem]{Corollary}
\newtheorem{fact}[theorem]{Fact}
\newtheorem{conjecture}[theorem]{Conjecture}
\newtheorem*{main}{Main Theorem}

\theoremstyle{definition}
\newtheorem{definition}[theorem]{Definition}
\newtheorem{example}[theorem]{Example}
\newtheorem{algorithm}{Algorithm}

\theoremstyle{remark}
\newtheorem{remark}{Remark}
\newtheorem{claim}{Claim}

\newcommand{\N}{\mathbb{N}}
\newcommand{\R}{\mathbb{R}}
\newcommand{\Cdot}{\boldsymbol{\cdot}}

\begin{abstract}
We pursue the investigation of the concept of $n$-distance, an $n$-variable version of the classical concept of distance recently introduced and investigated by Kiss, Marichal, and Teheux. We especially focus on the challenging problem of computing the best constant associated with a given $n$-distance. In particular, we define and investigate the best constants related to partial simplex inequalities. We also introduce and discuss some subclasses of $n$-distances defined by considering some properties. Finally, we discuss an interesting link between the concepts of $n$-distance and multidistance.
\end{abstract}

\keywords{$n$-distance, standard $n$-distance, multidistance, simplex inequality}

\subjclass[2010]{Primary 39B72; Secondary 26D99}

\maketitle

\section{Introduction}

Let $X$ be an arbitrary set, with $|X|\geq 2$, let $n\geq 2$ be an integer, and set $\R_+=\left[0,+\infty\right[$. Recall that a map $d\colon X^n\to\R_+$ is said to be an \emph{$n$-distance} (a \emph{distance} if $n=2$) on $X$ if it satisfies the following three conditions:
\begin{enumerate}
\item[(i)] $d(x_1,\ldots,x_n)=0$ if and only if $x_1=\cdots = x_n$,
\item[(ii)] $d$ is invariant under any permutation of its arguments,
\item[(iii)] $d(x_1,\ldots,x_n)\leq\sum_{i=1}^nd(x_1,\ldots,x_n)_i^z$ for all $x_1,\ldots,x_n,z\in X$.
\end{enumerate}
Here and throughout, the notation $d(x_1,\ldots,x_n)_i^z$ stands for the function obtained from $d(x_1,\ldots,x_n)$ by setting its $i$th variable to $z$. Condition (iii) is referred to as the \emph{simplex inequality} (the \emph{triangle inequality} if $n=2$).

In the special case of $n=3$, the concept of $n$-distance was introduced in 1992 by Dhage \cite{Dhage1992} and called $D$-metrics. The general $n$-ary version defined above seems to be introduced only recently by Kiss et al.~\cite{KisMarTeh16,KisMarTeh18}.

We also observe that various alternative proposals for $n$-variable distances have been introduced so far, each of those having interesting features (see, e.g., Deza and Deza~\cite[Chapter 3]{Deza2014}).

In this paper, we focus our investigation on $n$-distances and particularly on the following remarkable property of $n$-distances. For many $n$-distances, the simplex inequality can be refined into
\begin{equation}\label{bestKn}
d(x_1,\ldots,x_n)~\leq ~K_n{\,}\sum_{i=1}^nd(x_1,\ldots,x_n)_i^z,\qquad x_1,\ldots,x_n,z\in X,
\end{equation}
for some constant $K_n\in\left]0,1\right[$. To give an example, the \emph{cardinality based} $n$-distance defined by $d(x_1,\ldots,x_n)=|\{x_1,\ldots,x_n\}|-1$ satisfies \eqref{bestKn} with $K_n=(n-1)^{-1}$ and this constant is \emph{optimal} in the sense that \eqref{bestKn} no longer holds if $K_n$ is replaced by any value lower than $(n-1)^{-1}$. In fact, the constant $K_n=(n-1)^{-1}$ is attained, e.g., when $x_1\neq x_2=\cdots =x_n=z$.

It is important to note that condition (i) above is necessary for such a constant to exist in $\left]0,1\right[$. Indeed, as it was observed in \cite{KisMarTeh18}, we always have $K_n=1$ if condition (i) is replaced by the following one:
\begin{enumerate}
\item[(i')] $d(x_1,\ldots,x_n)=0$ if and only if $|\{x_1,\ldots,x_n\}|<n$.
\end{enumerate}

The main purpose of paper \cite{KisMarTeh18} was to provide several instances of $n$-distances together with their associated best (i.e., optimal) constants (see Example~\ref{ex:q} and Table~\ref{tab:q} below), with the additional objective of pointing out relevant properties of $n$-distances.

In the present paper, we further investigate the general properties of the best constants associated with $n$-distances. More precisely, in Section 2 we recall some instances of $n$-distances and provide a few new ones. We observe that the best constant associated with any $n$-distance cannot be lower than $(n-1)^{-1}$ and we introduce the class of those $n$-distances, that we call ``standard $n$-distances'', whose associated best constants have precisely the value $(n-1)^{-1}$. We also investigate the problem of constructing an $n$-distance with a prescribed best constant. In Section 3, we show that many $n$-distances satisfy partial simplex inequalities, i.e., simplex inequalities whose sums have less than $n$ summands. We also investigate the best constants related to these partial simplex inequalities. In Section 4, we introduce and investigate subclasses of $n$-distances defined by considering additional properties such as repetition invariance and nonincreasingness under identification of variables. Finally, in Section 5, we show how some standard $n$-distances can be used to define multidistances, which are special multi-argument distances introduced by Mart\'{\i}n and Mayor \cite{MarMay11}.

The investigation of $n$-distances and of the associated best constants seems to be very recent and needs much more examples to be better understood. We hope that by providing some examples and results here we might attract researchers and make this exciting topic better known.

To make the reading of the paper easier, we have postponed the proofs of most of our results to the Appendix.

\section{Definitions and examples}

Before presenting the formal definition of the concept of best constant associated with an $n$-distance, let us first recall some $n$-distances introduced and investigated in \cite{KisMarTeh18}.

\begin{example}[{see \cite{KisMarTeh18}}]\label{ex:q}
The following are instances of $n$-distances. Some of them are defined in terms of a given distance $d_2$ on $X$.
\begin{itemize}
\item Drastic $n$-distance
$$
d(x_1,\ldots,x_n) ~=~
\begin{cases}
0, & \text{if $x_1=\cdots =x_n$},\\
1, & \text{otherwise}.
\end{cases}
$$
\item Cardinality based $n$-distance
$$
d(x_1,\ldots,x_n)~=~|\{x_1,\ldots,x_n\}|-1.
$$
\item Diameter
$$
d(x_1,\ldots,x_n)~=~\max_{\{i,j\}\subseteq\{1,\ldots,n\}}d_2(x_i,x_j).
$$
\item Sum based $n$-distance
$$
d(x_1,\ldots,x_n)~=~\sum_{\{i,j\}\subseteq\{1,\ldots,n\}}d_2(x_i,x_j).
$$
\item Arithmetic mean based $n$-distance ($X=\R$)
$$
d(x_1,\ldots,x_n)~=~\frac{1}{n}\sum_{i=1}^nx_i-\min\{x_1,\ldots,x_n\}.
$$
\item Fermat $n$-distance
$$
d(x_1,\ldots,x_n)~=~\min_{x\in X}\sum_{i=1}^nd_2(x_i,x).
$$
\item Number of lines determined by $n$ points in $\R^2$ ($X=\R^2$).
\item Radius of the smallest circle enclosing $n$ points in $\R^2$ ($X=\R^2$).
\item Area of the smallest circle enclosing $n$ points in $\R^2$ ($X=\R^2$ and $n\geq 3$).
\end{itemize}
\end{example}

\begin{definition}[{see \cite{KisMarTeh18}}]
The \emph{best constant} associated with an $n$-distance $d$ on $X$ is the infimum $K_n^*$ of the set of real numbers $K_n\in\left]0,1\right]$ for which the condition
\begin{equation}\label{eq:Kndd}
d(x_1,\ldots,x_n)~\leq ~ K_n\sum_{i=1}^nd(x_1,\ldots,x_n)_i^z{\,},\qquad x_1,\ldots,x_n,z\in X,
\end{equation}
holds. Equivalently,
$$
K^*_n ~=~ \sup_{\textstyle{x_1,\ldots,x_n,z\in X\atop |\{x_1,\ldots,x_n\}|{\,}\geq{\,} 2}}\frac{d(x_1,\ldots,x_n)}{\sum_{i=1}^nd(x_1,\ldots,x_n)_i^z}{\,}.
$$
We say that the best constant $K^*_n$ is \emph{attained} if there exists $(x_1,\ldots,x_n;z)\in X^{n+1}$, with $|\{x_1,\ldots,x_n\}|\geq 2$, such that
$$
d(x_1,\ldots,x_n)~=~ K^*_n\sum_{i=1}^nd(x_1,\ldots,x_n)_i^z{\,}.
$$
\end{definition}

\begin{remark}
We always have $K^*_2=1$ and this constant is attained regardless of the distance $d$ considered on $X$. Indeed, we always have
$$
d(x_1,x_2) ~=~ \sum_{i=1}^2d(x_1,x_2)_i^{x_1},\qquad x_1,x_2\in X.
$$
\end{remark}

Table~\ref{tab:q} provides the best constants corresponding to most of the $n$-distances defined in Example~\ref{ex:q}. As observed in \cite{KisMarTeh18}, the search for the best constant associated with a given $n$-distance is usually not an easy problem. It is strongly dependent on the $n$-distance itself. In some cases, we could at most determine an interval in which the best constant lies. We also note that all the best constants that we have found thus far are attained.

\begin{table}[tbp]
\begin{center}
\begin{small}
\begin{tabular}{|l|c|c|}
\hline Name & Best constant & Type\\
\hline
Drastic $n$-distance & $K^*_n=(n-1)^{-1}$ & attained\\
Cardinality based $n$-distance & $K^*_n=(n-1)^{-1}$ & attained\\
Diameter & $K^*_n=(n-1)^{-1}$ & attained\\
Sum based $n$-distance & $K^*_n=(n-1)^{-1}$ & attained\\
Arithmetic mean based $n$-distance & $K^*_n=(n-1)^{-1}$ & attained\\
Radius of the smallest encl.\ circle & $K^*_n=(n-1)^{-1}$ & attained\\
Area of the smallest encl.\ circle ($n\geq 3$) & $K_n^*=(n-3/2)^{-1}$ & attained\\
Fermat $n$-distance & $K_n^*\leq (4n-4)/(3n^2-4n)$ & ?\\
Number of lines ($n\geq 3$) & $(n-2+2/n)^{-1}\leq K_n^*<(n-2)^{-1}$ & ?\\
\hline
\end{tabular}
\end{small}
\end{center}
\bigskip
\caption{Best constants associated with some $n$-distances}
\label{tab:q}
\end{table}

We now state a remarkable, although almost trivial, result showing that the best constant associated with any $n$-distance cannot be lower than $(n-1)^{-1}$.

\begin{proposition}\label{prop:nfkns}
For any $n$-distance $d$ on $X$, we have $K_n^*\geq (n-1)^{-1}$.
\end{proposition}

\begin{proof}
Let $d$ be an $n$-distance on $X$ and let $x,y\in X$, with $x\neq y$. Using \eqref{eq:Kndd}, we obtain
$$
d(x,y,\ldots,y) ~\leq ~ K_n\sum_{i=1}^nd(x,y,\ldots,y)_i^y ~=~ K_n(n-1){\,}d(x,y,\ldots,y).
$$
Since $d(x,y,\ldots,y)\neq 0$, we get $K_n\geq (n-1)^{-1}$.
\end{proof}

As we can see in Table~\ref{tab:q}, many of the $n$-distances satisfying $K_n^*=(n-1)^{-1}$ are based on rather natural constructions. This observation together with the previous proposition motivate the following terminology.

\begin{definition}
We say that an $n$-distance $d$ on $X$ is \emph{standard} if $K_n^*=(n-1)^{-1}$.
\end{definition}

We know for instance that the radius (or equivalently, the diameter) of the smallest enclosing circle in $\R^2$ defines a standard $n$-distance (see Table~\ref{tab:q}). It is easy to see that this is also the case for the diameter of the smallest enclosing Chebyshev ball in $\R^q$ for any integer $q\geq 2$, that is,
$$
d(x_1,\ldots,x_n) ~=~ \max_{\{i,j\}\subseteq\{1,\ldots,n\}}\|x_i-x_j\|_{\infty}{\,},\qquad x_1,\ldots x_n\in\R^q.
$$
Indeed, this is exactly the diameter-type $n$-distance (see Example~\ref{ex:q}) constructed from the Chebyshev distance on $\R^q$, namely
$$
d_2(x,y) ~=~ \|x-y\|_{\infty}{\,},\qquad x,y\in\R^q.
$$

A few nonstandard $n$-distances based on geometric constructions have also been defined and investigated in \cite{KisMarTeh18} (see Table~\ref{tab:q}). Now, it may be an interesting challenge to introduce further nonstandard $n$-distances and determine their associated best constants. In the following result, we provide such an $n$-distance on $\R$.

\begin{proposition}[Length of a largest inner interval]\label{prop:Llii}
Let $d\colon\R^n\to\R_+$ be the map defined by
\begin{equation}\label{eq:LLII1}
d(x_1,\ldots,x_n) ~=~ \max_{i=1,\ldots,n-1}(x_{(i+1)}-x_{(i)}),
\end{equation}
where the symbol $x_{(i)}$ stands for the $i$th smallest element among $x_1,\ldots,x_n$. Then $d$ is an $n$-distance on $\R$. Its best constant is $K^*_n=2/n$ and is attained at any $(x_1,\ldots,x_n;z)$ such that $x_1< x_2=\cdots =x_n$ and $z=(x_1+x_2)/2$.
\end{proposition}

\begin{remark}
It is not difficult to see that, for any integer $p\geq 1$, the map obtained by raising the $n$-distance defined in \eqref{eq:LLII1} to the $p$th power is an $n$-distance on $\R$ if and only if $n\geq 2^p$. Its associated best constant is $K^*_n=2^p/n$ and is attained at any $(x_1,\ldots,x_n;z)$ such that $x_1< x_2=\cdots =x_n$ and $z=(x_1+x_2)/2$. The proof is similar to that of Proposition~\ref{prop:Llii}.
\end{remark}

There must be plenty of natural ways to construct maps that look like $n$-distances. However, for some of them it might be tricky to establish that they are genuine $n$-distances and find their associated best constants.

We end this section by discussing the challenging problem of constructing an $n$-distance with a prescribed best constant.

Given a real number $s\in [(n-1)^{-1},1]$, it is natural to ask whether there exists an $n$-distance whose best constant has exactly the value $s$. The following proposition answers this question in the affirmative by providing one-parameter families of $n$-distances covering all possible best constants.

\begin{proposition}\label{prop:Ksns}
We assume that $|X|\geq 3$. Let $s\in [(n-1)^{-1},1]$, let $e\in X$, and let $d$ be any standard $n$-distance on $X$. Let also
$$
C_{n,s} ~=~ \frac{1}{s}\sup_{x_1,\ldots,x_n\in X\setminus\{e\}}\frac{d(x_1,\ldots,x_n)}{\sum_{i=1}^nd(x_1,\ldots,x_n)_i^e}~>~0.
$$
Then the map $d_s\colon X^n\to\R_+$ defined by
$$
d_s(x_1,\ldots,x_n) ~=~
\begin{cases}
C_{n,s}{\,}d(x_1,\ldots,x_n), & \text{if $e\in\{x_1,\ldots,x_n\}$},\\
d(x_1,\ldots,x_n), & \text{otherwise},
\end{cases}
$$
is an $n$-distance on $X$ whose best constant is $K^*_n=s$.
\end{proposition}

\section{Partial simplex inequalities}

It is natural to ask whether a given $n$-distance $d$ on $X$ (with $n\geq 3$) satisfies a partial simplex inequality, i.e., an inequality of the form
\begin{equation}\label{eq:partkk}
d(x_1,\ldots,x_n)~\leq ~K_{n,k}\sum_{i=1}^kd(x_1,\ldots,x_n)_i^z{\,},\quad x_1,\ldots,x_n,z\in X,
\end{equation}
for some $k\in\{2,\ldots,n-1\}$ and some $K_{n,k}>0$. When $K_{n,k}\leq 1$, such an inequality simply means that any $k$-section of $d$ (obtained from $d$ by fixing $n-k$ of its variables) satisfies the simplex inequality (the triangle inequality if $k=2$).

We observe that inequality \eqref{eq:partkk} does not make sense when $k=1$. Indeed, for any distinct $x,z\in X$, we would have
$$
0 ~<~ d(x,z,\ldots,z) ~\leq ~ K_{n,1}{\,}d(z,z,\ldots,z) ~=~ 0,
$$
a contradiction.

The following proposition shows that \eqref{eq:partkk} holds for all $n$-distances whose associated best constants $K^*_n$ satisfy $K^*_n<(n-k)^{-1}$.

\begin{proposition}\label{prop:asww7df}
For any $n$-distance $d$ on $X$ and any integer $k$ satisfying $n-1/K^*_n<k\leq n$, we have
$$
d(x_1,\ldots,x_n)~\leq ~\frac{1}{1/K^*_n-n+k}{\,}\sum_{i=1}^kd(x_1,\ldots,x_n)_i^z{\,},\qquad x_1,\ldots,x_n,z\in X.
$$
Moreover, for any $x_1,\ldots,x_n,z\in X$, we have
\begin{equation}\label{eq:partkkeq}
d(x_1,\ldots,x_n)~=~\frac{1}{1/K^*_n-n+k}{\,}\sum_{i=1}^kd(x_1,\ldots,x_n)_i^z
\end{equation}
if and only if $K^*_n$ is attained at $(x_1,\ldots,x_n;z)$ and
\begin{equation}\label{eq:partkkeq2}
d(x_1,\ldots,x_n) ~=~ d(x_1,\ldots,x_n)_i^z{\,},\qquad i\in\{k+1,\ldots,n\}.
\end{equation}
\end{proposition}

\begin{corollary}\label{cor:asww7dfc}
Let $d$ be an $n$-distance on $X$ and let $x_1,\ldots,x_n,z\in X$. If \eqref{eq:partkkeq} holds for some integer $k=p$ satisfying $n-1/K^*_n<p\leq n$, then it holds also for any $k\in\{p,\ldots,n\}$.
\end{corollary}

In view of the first part of Proposition~\ref{prop:asww7df}, we naturally consider the following definition.

\begin{definition}\label{de:bestkk5}
Let $d$ be an $n$-distance on $X$ and let $k\in\{2,\ldots,n\}$. Assume that the set of real numbers $K_{n,k}>0$ for which the condition
$$
d(x_1,\ldots,x_n)~\leq ~K_{n,k}\sum_{i=1}^kd(x_1,\ldots,x_n)_i^z,\quad x_1,\ldots,x_n,z\in X,~|\{x_1,\ldots,x_n\}|\geq 2,
$$
holds is nonempty. Then, the infimum $K^*_{n,k}$ of this set is called the \emph{best $k$-constant} associated with $d$.
\end{definition}

\begin{example}\label{ex:ccc7}
Consider the $n$-distance ($n\geq 3$) defined by the area of the smallest enclosing circle in $\R^2$. Its associated best constant is $K^*_n=(n-3/2)^{-1}$ (see Table~\ref{tab:q}). By Proposition~\ref{prop:asww7df}, we immediately see that $K^*_{n,k}$ exists and satisfies
$$
K^*_{n,k} ~\leq ~ \frac{1}{1/K^*_n-n+k} ~=~ \frac{1}{k-3/2}
$$
for any $k\in\{2,\ldots,n\}$. Now, we know \cite{KisMarTeh18} that $K^*_n$ is attained at any tuple $(x_1,\ldots,x_n;z)$ such that $x_1\neq x_2$ and $x_3=\cdots =x_n=z=(x_1+x_2)/2$. For such a tuple, condition \eqref{eq:partkkeq2} clearly holds for any $k\in\{2,\ldots,n\}$, and hence condition \eqref{eq:partkkeq} also holds for any $k\in\{2,\ldots,n\}$. Using the second part of Proposition~\ref{prop:asww7df}, we finally obtain that
$$
K^*_{n,k} ~=~ \frac{1}{1/K^*_n-n+k} ~=~ \frac{1}{k-3/2}
$$
for any $k\in\{2,\ldots,n\}$.
\end{example}

\begin{example}\label{ex:2kfff}
If $d$ is the length of a largest inner interval as defined in Proposition~\ref{prop:Llii}, then it is not difficult to show that $K^*_{n,k}=2/k$ for any $k\in\{2,\ldots,n\}$ (just proceed as in the proof of Proposition~\ref{prop:Llii}). This example also shows that there are nonstandard $n$-distances for which $K^*_{n,2}=1$ (i.e., any $2$-section of $d$ satisfies the triangle inequality).
\end{example}

Proceeding as in the proof of Proposition~\ref{prop:nfkns} and using Proposition~\ref{prop:asww7df}, we can easily derive the following result.

\begin{proposition}\label{prop:nfknskk}
Let $d$ be an $n$-distance on $X$. If $K^*_{n,k}$ exists for some $k=p\in\{2,\ldots,n\}$, then it exists also for any $k\in\{p,\ldots,n\}$ and we have $K^*_{n,k}\geq (k-1)^{-1}$ for any $k\in\{p,\ldots,n\}$. Moreover, if $d$ is standard, then $K^*_{n,k}$ exists for any $k\in\{2,\ldots,n\}$ and we have $K^*_{n,k}=(k-1)^{-1}$ for any $k\in\{2,\ldots,n\}$.
\end{proposition}

Proposition~\ref{prop:nfknskk} shows that for any standard $n$-distance and any $k\in\{2,\ldots,n\}$, we have
$$
d(x_1,\ldots,x_n)~\leq ~\frac{1}{k-1}{\,}\sum_{i=1}^kd(x_1,\ldots,x_n)_i^z{\,},\qquad x_1,\ldots,x_n,z\in X,
$$
and the constant $K^*_{n,k}=(k-1)^{-1}$ is the lowest possible constant that we can reach over all the $n$-distances. Also, the above inequality implies that for any integer $k\in\{2,\ldots,n\}$, we have
$$
d(\underbrace{x,\ldots,x}_{k},z,\ldots,z) ~\leq ~ \frac{k}{k-1}{\,}d(\underbrace{x,\ldots,x}_{k-1},z,\ldots,z),\qquad x,z\in X.
$$

From the results and examples above, we obtain the following proposition, which provides inequality conditions on the best constants $K^*_n$ and $K^*_{n,k}$.

\begin{proposition}\label{prop:37fgsd}
For any $n$-distance $d$ on $X$ and any integer $k$ satisfying $n-1/K^*_n<k\leq n$, the number $K^*_{n,k}$ exists and satisfies the inequalities
$$
\frac{1}{k-1}~\leq ~ K^*_{n,k} ~\leq ~ \frac{1}{1/K^*_n-n+k}
$$
and
$$
K^*_n ~\geq ~ \frac{1}{1/K^*_{n,k}+n-k}~\geq ~\frac{1}{n-1}{\,}.
$$
All these inequalities are equalities if $d$ is standard. However, they may be strict for some nonstandard $n$-distances.
\end{proposition}

\begin{proof}
The existence of $K^*_{n,k}$ and the first two inequalities follow from Propositions~\ref{prop:asww7df} and \ref{prop:nfknskk}. The remaining two inequalities follow immediately. The case where $d$ is standard follows from Proposition~\ref{prop:nfknskk}. The strictness of the inequalities occurs for instance when considering the $n$-distance discussed in Example~\ref{ex:2kfff} (taking $n>k>\max\{2,n/2\}$).
\end{proof}

The following result provides sufficient conditions for an $n$-distance to be standard. It is particularly interesting for $n=3$, where condition (b) reduces to requiring that any $2$-section of $d$ satisfies the triangle inequality.

\begin{proposition}\label{prop:suffcondst9}
Let $k\in\{2,\ldots,n-1\}$ and let $d$ be an $n$-distance on $X$ satisfying the following two conditions.
\begin{enumerate}
\item[(a)] $K^*_n<(n-k)^{-1}$ and is attained at some $(x_1,\ldots,x_n;z)$ satisfying condition ~\eqref{eq:partkkeq2}.
\item[(b)] Condition \eqref{eq:partkk} holds for $K_{n,k}=(k-1)^{-1}$.
\end{enumerate}
Then $d$ is standard.
\end{proposition}

\begin{proof}
Using the second part of Proposition~\ref{prop:asww7df}, we obtain $K^*_{n,k} = (1/K^*_n-n+k)^{-1}$. Since $K^*_{n,k}=(k-1)^{-1}$ by condition (b) and Proposition~\ref{prop:nfknskk}, we derive $K^*_n=(n-1)^{-1}$.
\end{proof}

By applying a symmetrization technique on the partial simplex inequality, we obtain the following result, which provides an additional inequality condition on the best constants $K^*_n$ and $K^*_{n,k}$.

\begin{proposition}\label{prop:symm5}
Let $d$ be an $n$-distance on $X$ and let $k\in\{2,\ldots,n\}$. If $K^*_{n,k}$ exists, then we have $K^*_n\leq\frac{k}{n}K^*_{n,k}{\,}$ and the constant $\frac{k}{n}$ is optimal (in the sense that the equality holds for at least one $n$-distance).
\end{proposition}

In the following example, which is a continuation of Proposition~\ref{prop:Ksns}, we provide the best $k$-constant for any $k\in\{2,\ldots,n\}$ of an $n$-distance that has a prescribed best constant $K^*_n\in [(n-1)^{-1},1]$.

\begin{example}\label{ex:2kfffr}
For any $s\in [(n-1)^{-1},1]$ and any $e\in X$, if we choose the drastic distance for $d$ in Proposition~\ref{prop:Ksns}, then the map $d_s$ is defined by
$$
d_s(x_1,\ldots,x_n) ~=~
\begin{cases}
\frac{1}{sn}{\,}d(x_1,\ldots,x_n), & \text{if $e\in\{x_1,\ldots,x_n\}$},\\
d(x_1,\ldots,x_n), & \text{otherwise}.
\end{cases}
$$
Let $K^*_{n,k}$ be the best $k$-constant associated with $d_s$. By proceeding as in the proof of Proposition~\ref{prop:Ksns}, we see that $K^*_{n,k}$ exists for any $k\in\{2,\ldots,n\}$ and we have $K^*_{n,k} = \max\{n s/k,(k-1)^{-1}\}$. If $s\geq k/(n(k-1))$, then $K^*_{n,k}=n s/k$, which illustrates again the optimality of the constant $k/n$ in Proposition~\ref{prop:symm5} (since $K^*_n=s$). If $s\leq k/(n(k-1))$, then we have $K^*_{n,k}=(k-1)^{-1}$ even when $d$ is not standard (which shows that the converse of the second part of Proposition~\ref{prop:nfknskk} does not hold).
\end{example}

We observed that the second displayed inequality in Proposition~\ref{prop:37fgsd} reduces to an equality when considering Example~\ref{ex:ccc7}. Now, in the following proposition we provide an example of an $n$-distance for which this equality holds for any $k\in\{2,\ldots,n\}$ and that has a prescribed best constant $K^*_n\in\left[(n-1)^{-1},(n-2)^{-1}\right[$.

\begin{proposition}\label{prop:Ksnsab}
We assume that $|X|\geq 4$. Let $s\in\left[(n-1)^{-1},(n-2)^{-1}\right[$ and set $C_{n,s}=\frac{2}{1/s-n+2}\geq 2$. Let also $a,b\in X$, $a\neq b$, and let $d$ be the drastic $n$-distance on $X$. Then, the map $d_s\colon X^n\to\R_+$ defined by
$$
d_s(x_1,\ldots,x_n) ~=~
\begin{cases}
C_{n,s}{\,}d(x_1,\ldots,x_n), & \text{if $a,b\in\{x_1,\ldots,x_n\}$},\\
d(x_1,\ldots,x_n), & \text{otherwise},
\end{cases}
$$
is an $n$-distance on $X$ whose best constant is $K^*_n=s$. Moreover, for any $k\in\{2,\ldots,n\}$, its best $k$-constant $K^*_{n,k}$ exists and we have $K^*_{n,k}=\frac{1}{1/K^*_n-n+k}$.
\end{proposition}

\begin{remark}
We observe that Proposition~\ref{prop:Ksnsab} can be easily generalized by considering the assumption that $s\in\left[(n-\ell)^{-1},(n-\ell -1)^{-1}\right[$ for some integer $\ell$ satisfying $1\leq\ell\leq n-2$. However, the value of $C_{n,s}$ and the admissible range of $k$ are to be adapted accordingly.
\end{remark}

\section{Additional properties for $n$-distances}

We now introduce and investigate subclasses of $n$-distances defined by considering some special properties. Throughout this section, for any $k\in\{1,\ldots,n\}$ and any $x\in X$, the notation $k\Cdot x$ stands for the $k$-tuple $x,\ldots,x$. For instance, we have
$$
d(3\Cdot x,2\Cdot y) ~=~ d(x,x,x,y,y).
$$

\begin{definition}\label{de:ssi7}
Let $k\in\{2,\ldots,n\}$. We say that an $n$-distance $d$ on $X$ fulfills the \emph{strong $k$-simplex inequality} if there exists $M_{n,k}>0$ such that, for any $n_1,\ldots,n_k\in\{1,\ldots,n\}$ with $n_1+\cdots +n_k=n$, the map $d'\colon X^k\to\R_+$ defined by
\begin{equation}\label{eq:dprime}
d'(x_1,\ldots,x_k) ~=~ d(n_1\!\Cdot x_1,\ldots,n_k\!\Cdot x_k),\qquad x_1,\ldots,x_k\in X,
\end{equation}
satisfies
$$
d'(x_1,\ldots,x_k) ~\leq ~M_{n,k}{\,}\sum_{i=1}^k d'(x_1,\ldots,x_k)_i^z{\,},\qquad x_1,\ldots,x_k,z\in X.
$$
\end{definition}

For instance, a $4$-distance $d$ on $X$ satisfies the strong $2$-simplex inequality if and only if there exists $M>0$ such that
\begin{eqnarray*}
d(2\Cdot x_1,2\Cdot x_2) &\leq & {\!}M{\,}(d(2\Cdot z,2\Cdot x_2)+d(2\Cdot x_1,2\Cdot z)),\\
d(3\Cdot x_1,x_2) &\leq & {\!}M{\,}(d(3\Cdot z,x_2)+d(3\Cdot x_1,z)),
\end{eqnarray*}
for all $x_1,x_2,z\in X$.

\begin{remark}
It should be noted that the map $d'\colon X^k\to\R_+$ defined in \eqref{eq:dprime} need not be symmetric. In particular, $d'$ need not be a $k$-distance.
\end{remark}

Many $n$-distances discussed in this paper satisfy the strong $k$-simplex inequality for $k=2,\ldots,n$. This is the case for instance for the radius of the smallest enclosing circle in $\R^2$, where the map $d'\colon X^k\to\R_+$ defined in \eqref{eq:dprime} is symmetric. The following example shows that the arithmetic mean based $n$-distance also satisfies the strong $k$-simplex inequality for $k=2,\ldots,n$. However, in this latter case the map $d'$ is clearly not symmetric.

\begin{example}\label{ex:ambRED4}
For any $k\in\{2,\ldots,n\}$, the arithmetic mean based $n$-distance satisfies the strong $k$-simplex inequality with constant $M_{n,k}=(k-1)^{-1}$. Indeed, let $n_1,\ldots,n_k\in\{1,\ldots,n\}$ be such that $n_1+\cdots +n_k=n$ and let $d'\colon X^k\to\R_+$ be the map defined in \eqref{eq:dprime}. Let $x_1,\ldots,x_k\in X$. We can assume without loss of generality that $x_1\leq\cdots\leq x_k$. Then the inequality
$$
d'(x_1,\ldots,x_k)~\leq ~\frac{1}{k-1}\sum_{i=1}^kd'(x_1,\ldots,x_k)_i^z,
$$
reduces to
$$
\frac{1}{n}\sum_{i=1}^kn_ix_i-x_1 ~\leq ~ \frac{1}{n}\sum_{i=1}^kn_ix_i-\min\{x_1,z\}+\frac{1}{k-1}{\,}(z-\min\{x_2,z\}){\,},
$$
or equivalently,
$$
(k-1)(x_1-\min\{x_1,z\})+(z-\min\{x_2,z\})~\geq ~0,
$$
which clearly holds.
\end{example}

\begin{definition}
We say that an $n$-distance $d$ on $X$ is \emph{repetition invariant} if for any
$x_1$, $\ldots${\,}, $x_n$, $x'_1$, $\ldots${\,}, $x'_n\in X$, we have
$$
\{x_1,\ldots,x_n\}~=~\{x'_1,\ldots,x'_n\}\quad\Rightarrow\quad d(x_1,\ldots,x_n)~=~d(x'_1,\ldots,x'_n).
$$
\end{definition}

Many $n$-distances discussed in this paper are repetition invariant. For instance, the cardinality based $n$-distance and the length of a largest inner interval (see Proposition~\ref{prop:Llii}) are repetition invariant. The following example provides two instances of $n$-distances that are not repetition invariant.

\begin{example}\label{ex:ambFeIR5}
The arithmetic mean based $n$-distance $d$ on $\R$ is not repetition invariant. Indeed, for any $x,y\in\R$ such that $x< y$, we have $d(x,y,y)=\frac{2}{3}(y-x)>\frac{1}{3}(y-x)=d(x,x,y)$. Also, in general, the Fermat $n$-distance $d$ on $X$ defined in terms of a distance $d_2$ on $X$ is not repetition invariant. Indeed, consider the distance $d_2(x,y)=|x-y|$ on $X=\R$. Denoting the Fermat point of $(0,0,1,1)$ by $x\in [0,1]$, we have $d(0,0,1,1)=2x+2(1-x)=2$. Denoting the Fermat point of $(0,1,1,1)$ by $x\in [0,1]$, we have $d(0,1,1,1)=3-2x$, which is minimized by $x=1$, with value $1$.
\end{example}

\begin{fact}\label{fact:1}
If an $n$-distance $d$ on $X$ is repetition invariant and satisfies the strong $k$-simplex inequality for some $k\in\{2,\ldots,n\}$, then the map $d'\colon X^k\to\R_+$ defined in \eqref{eq:dprime} is symmetric, and hence it is a $k$-distance whenever $M_{n,k}\leq 1$. As an important special case when $k=2$, the set $X$ is metrizable by $d'$ whenever $M_{n,2}\leq 1$.
\end{fact}

In the next proposition, we show that any standard $n$-distance that is repetition invariant satisfies the strong $k$-simplex inequality for any $k\in\{2,\ldots,n\}$. We also provide the optimal value of the corresponding constant $M_{n,k}$. Of course the case $k=n$ is trivial and we clearly have $M_{n,n}=(n-1)^{-1}$. We first present a technical lemma.

\begin{lemma}\label{lemma:f456dfgg}
Let $d$ be a standard $n$-distance on $X$ that is repetition invariant, let $k\in\{2,\ldots,n-1\}$, and let $p\in\{0,\ldots,n-k\}$. Then, for any $x_1,\ldots,x_k,z\in X$, we have
\begin{multline*}
d(x_1,\ldots,x_{k-1},x_k,\ldots,x_k)\\
\leq ~ \frac{k+p}{(k-1)(k+p-1)}{\,}\sum_{i=1}^{k-1}d(x_1,\ldots,x_{k-1},x_k,\ldots,x_k)_i^z \\
\null + \frac{p+1}{(k-1)(k+p-1)}{\,}d(x_1,\ldots,x_{k-1},z,\ldots,z).
\end{multline*}
Moreover, the constants are optimal.
\end{lemma}

\begin{proposition}\label{prop:azzs54sfd}
If a standard $n$-distance $d$ on $X$ is repetition invariant, then, for any $k\in\{2,\ldots,n-1\}$, it satisfies the strong $k$-simplex inequality with constant
$$
M_{n,k} ~=~ (k-1)^{-1}+(k(k-1)(n-1))^{-1}
$$
and this constant is optimal and attained.
\end{proposition}

\begin{remark}
By Proposition~\ref{prop:azzs54sfd}, any standard $n$-distance $d$ on $X$ that is repetition invariant satisfies the strong $k$-simplex inequality for any $k\in\{2,\ldots,n\}$. Moreover, if $k\geq 3$, then we have $M_{n,k}\leq 1$, and hence the map $d'$ defined in \eqref{eq:dprime} is a $k$-distance by Fact~\ref{fact:1}. We might then say that $d$ is \emph{reducible} to a $k$-distance.
\end{remark}

We observe that by removing the standardness assumption in Lemma~\ref{lemma:f456dfgg} and Proposition~\ref{prop:azzs54sfd}, we can prove similarly the following more general result. However, the optimality of $M_{n,k}$ is no longer ensured.

\begin{proposition}
If an $n$-distance $d$ on $X$ is repetition invariant, then, for any integer $k$ satisfying $n-1/K^*_n<k<n$, it satisfies the strong $k$-simplex inequality with constant
$$
M_{n,k} ~=~ \frac{K^*_n+1}{1/K^*_n-n+k}-\frac{K^*_n}{k}{\,}.
$$
Also, we have $M_{n,k}\leq 1$ if $k\geq n+2-1/K^*_n$ (and $M_{n,k}>1$ if $k\leq n+1-1/K^*_n$).
\end{proposition}

In the following definition, we introduce a property for $n$-distances that is stronger than repetition invariance. We observe that this property was already considered in the special case of $n=3$ in the framework of $G$-metric spaces in \cite{Mustafa2006}.

\begin{definition}
We say that an $n$-distance $d$ on $X$ is \emph{nonincreasing under identification of variables} if for any $x_1,\ldots,x_n\in X$, we have
$$
d(x_1,\ldots,x_{n-1},x_n) ~\geq ~d(x_1,\ldots,x_{n-1},x_1).
$$
In other words, the distance cannot increase when two variables are identified.
\end{definition}

\begin{proposition}\label{prop:sfda4}
If an $n$-distance $d$ on $X$ is nonincreasing under identification of variables, then it is repetition invariant.
\end{proposition}

\begin{example}
The cardinality based $n$-distance on $X$ is nonincreasing under identification of variables. The length $d$ of a largest inner interval (see Proposition~\ref{prop:Llii}) is not nonincreasing under identification of variables. Indeed, we have $1=d(1,2,3)<d(1,3,3)=2$.
\end{example}

\begin{proposition}\label{prop:sfda45}
Let $d$ be an $n$-distance on $X$ that is nonincreasing under identification of variables. Let $k\in\{2,\ldots,n\}$ be such that $K^*_{n,k}$ exists (see Definition~\ref{de:bestkk5}). Then $d$ satisfies the strong $k$-simplex inequality with constant $M_{n,k}= K^*_{n,k}$.
\end{proposition}

\begin{corollary}
Let $d$ be a standard $n$-distance on $X$ that is nonincreasing under identification of variables. Then, for any $k\in\{2,\ldots,n\}$, $d$ satisfies the strong $k$-simplex inequality with constant $M_{n,k}= (k-1)^{-1}$.
\end{corollary}

\section{Multidistances}

Recall that a \emph{multidistance} on $X$, as defined by Mart\'{\i}n and Mayor \cite{MarMay11}, is a function $d\colon\bigcup_{n\geq 2}X^n\to\R_+$ satisfying the following three conditions, for every integer $n\geq 1$:
\begin{enumerate}
\item[(i)] $d(x_1,\ldots,x_n)=0$ if and only if $x_1=\cdots = x_n$,
\item[(ii)] $d|_{X^n}$ is invariant under any permutation of its arguments,
\item[(iii)] $d(x_1,\ldots,x_n)\leq\sum_{i=1}^nd(x_i,z)$ for all $x_1,\ldots,x_n,z\in X$.
\end{enumerate}

As already observed in \cite{KisMarTeh18}, some $n$-distances cannot be considered to define multidistances. For instance, the area of the smallest circle enclosing $n$ points in $\R^2$ cannot be used to define a multidistance (since the triangle inequality does not hold when $n=2$). Likewise, the number of lines determined by $n$ points of $\R^2$ cannot be used to define a multidistance. Indeed, if the points $x_1,\ldots,x_n$ are pairwise distinct and placed on a circle centered at $z$, then for any integer $n\geq 3$ we have
$$
d_n(x_1,\ldots,x_n) ~=~ {n\choose 2} ~ > ~ n ~=~ \sum_{i=1}^nd_2(x_i,z).
$$

We also observe that many $n$-distances can be considered to define multidistances, even in the nonstandard case. The largest inner interval as defined in Proposition~\ref{prop:Llii} could serve as a very simple example here.

In the following proposition, we show how multidistances can be easily defined from certain standard $n$-distances.

\begin{lemma}\label{lemma:sad7f}
Let $d$ be a standard $n$-distance on $X$, let $g\colon X^2\to\R_+$ be a function, and let $z\in X$. If $d(x,z,\ldots,z)\leq g(x,z)$ for all $x\in X$, then for any $k\in\{1,\ldots,n\}$ we have
$$
d(x_1,\ldots,x_k,z,\ldots,z) ~\leq ~ \sum_{i=1}^k g(x_i,z){\,},\qquad x_1,\ldots,x_k\in X.
$$
In particular,
$$
d(x_1,\ldots,x_n) ~\leq ~ \sum_{i=1}^n g(x_i,z){\,},\qquad x_1,\ldots,x_n\in X.
$$
\end{lemma}

\begin{proposition}\label{prop:mul55}
Let $(d_n)_{n\geq 2}$ be a sequence, where $d_n$ is a standard $n$-distance on $X$ satisfying
$$
d_n(x_n,z_n,\ldots,z_n)~\leq ~ d_2(x_n,z_n),\qquad n\geq 2{\,};~x_n,z_n\in X.
$$
Then the map $d\colon\bigcup_{n\geq 2}X^n\to\R_+$ defined by $d|_{X^n}=d_n$ is a multidistance on $X$.
\end{proposition}

\begin{proof}
This is an immediate consequence of Lemma~\ref{lemma:sad7f}.
\end{proof}

It is known \cite{AguMarMaySunVal12,MarMayVal11} that the radius of the smallest enclosing circle in $\R^2$ defines a multidistance. The next example shows how we can retrieve this result from Proposition~\ref{prop:mul55}.

\begin{example}
Consider the sequence $(d_n)_{n\geq 2}$, where the $n$-distance $d_n$ is defined by the radius of the smallest circle enclosing $n$ points in $\R^2$. We then have $d_n(x_n,z_n,\ldots,z_n)=d_2(x_n,z_n)$ for all $n\geq 2$ and all $x_n,z_n\in\R^2$. By Proposition~\ref{prop:mul55}, the map $d\colon\bigcup_{n\geq 2}(\R^2)^n\to\R_+$ defined by $d|_{X^n}=d_n$ for all $n\geq 2$ is a multidistance on $\R^2$.
\end{example}

The following example shows that the arithmetic mean based $n$-distance defines a multidistance, provided its binary version is doubled.

\begin{example}
Consider the sequence $(d_n)_{n\geq 2}$, where $d_n$ is the arithmetic mean based $n$-distance on $\R$. For any $n\geq 3$, we have $d_n(x_n,z_n,\ldots,z_n)\leq d_2(x_n,z_n)$ if and only if $z_n \leq x_n$. Now, replacing $d_2$ by the map $d'_2\colon\R^2\to\R_+$ defined by
$$
d'_2(x,z) ~=~ d_n(x,z,\ldots,z)+d_n(z,x,\ldots,x) ~=~ 2{\,}d_2(x,z),
$$
we obtain $d_n(x_n,z_n,\ldots,z_n)\leq d'_2(x_n,z_n)$ for all $n\geq 2$ and all $x_n,z_n\in\R$. By Proposition~\ref{prop:mul55}, the map $d\colon\bigcup_{n\geq 2}(\R^2)^n\to\R_+$ defined by $d|_{X^n}=d_n$ for all $n\geq 3$ and $d|_{X^2}=d'_2$ is a multidistance on $\R^2$.
\end{example}

In the following proposition, we show how $n$-distances can be defined from certain multidistances. The proof is straightforward and thus omitted.

\begin{proposition}
Let $d\colon\bigcup_{n\geq 2} X^n\to\R_+$ be a multidistance on $X$ and let $n\geq 3$ be an integer. If $d_n=d|_{X^n}$ is nonincreasing under identification of variables and satisfies $d(x,z)\leq d_n(x,z,\ldots,z)$ for all $x,z\in X$, then it is an $n$-distance.
\end{proposition}

\appendix
\section*{Appendix: proofs}

\begin{proof}[Proof of Proposition~\ref{prop:Llii}]
If $n=2$, then $d$ is the usual Euclidean distance on $\R$. Now suppose that $n\geq 3$ and let $x_1,\ldots,x_n,z\in\R$ be such that $|\{x_1,\ldots,x_n\}|\geq 2$. By symmetry, we can assume without loss of generality that $x_1\leq\cdots\leq x_n$. Let $p\in\{1,\ldots,n-1\}$ such that $d(x_1,\ldots,x_n)=x_{p+1}-x_p$. There are two exclusive cases to consider.
\begin{itemize}
\item Case $z\notin\left]x_p,x_{p+1}\right[$.
\begin{itemize}
\item If $1\neq p\neq n-1$, then $d(x_1,\ldots,x_n)_i^z\geq d(x_1,\ldots,x_n)$ for $i=1,\ldots,n$, that is,
$$
\sum_{i=1}^nd(x_1,\ldots,x_n)_i^z ~\geq ~ n{\,}d(x_1,\ldots,x_n).
$$
\item If $p=1$, then $d(x_1,\ldots,x_n)_i^z\geq d(x_1,\ldots,x_n)$ for $i=2,\ldots,n$, that is,
$$
\sum_{i=1}^nd(x_1,\ldots,x_n)_i^z ~\geq ~ (n-1){\,}d(x_1,\ldots,x_n).
$$
\item If $p=n-1$, then $d(x_1,\ldots,x_n)_i^z\geq d(x_1,\ldots,x_n)$ for $i=1,\ldots,n-1$, that is,
$$
\sum_{i=1}^nd(x_1,\ldots,x_n)_i^z ~\geq ~(n-1){\,}d(x_1,\ldots,x_n).
$$
\end{itemize}
\item Case $z\in\left]x_p,x_{p+1}\right[$. Set $\lambda=(z-x_p)/(x_{p+1}-x_p)$.
\begin{itemize}
\item If $1\neq p\neq n-1$, then $d(x_1,\ldots,x_n)_i^z\geq\max\{\lambda,1-\lambda\}{\,}d(x_1,\ldots,x_n)\geq \frac{1}{2}d(x_1,\ldots,x_n)$ for $i=1,\ldots,n$, that is,
$$
\sum_{i=1}^nd(x_1,\ldots,x_n)_i^z ~\geq~ \frac{n}{2}{\,}d(x_1,\ldots,x_n).
$$
\item If $p=1$, then $d(x_1,\ldots,x_n)_i^z\geq \max\{\lambda,1-\lambda\}{\,}d(x_1,\ldots,x_n)$ for $i=2,\ldots,n$, and $d(x_1,\ldots,x_n)_1^z\geq (1-\lambda){\,}d(x_1,\ldots,x_n)$, that is,
\begin{eqnarray*}
\lefteqn{\sum_{i=1}^nd(x_1,\ldots,x_n)_i^z}\\
&\geq &(n-1)\max\{\lambda,1-\lambda\}{\,}d(x_1,\ldots,x_n)+(1-\lambda){\,}d(x_1,\ldots,x_n)\\
&\geq &(n-2)\max\{\lambda,1-\lambda\}{\,}d(x_1,\ldots,x_n)+d(x_1,\ldots,x_n)\\
&\geq &\frac{n}{2}{\,}d(x_1,\ldots,x_n).
\end{eqnarray*}
\item If $p=n-1$, then $d(x_1,\ldots,x_n)_i^z\geq \max\{\lambda,1-\lambda\}{\,}d(x_1,\ldots,x_n)$ for $i=1,\ldots,n-1$, and $d(x_1,\ldots,x_n)_n^z\geq \lambda{\,}d(x_1,\ldots,x_n)$, that is,
\begin{eqnarray*}
\lefteqn{\sum_{i=1}^nd(x_1,\ldots,x_n)_i^z}\\
&\geq &(n-1)\max\{\lambda,1-\lambda\}{\,}d(x_1,\ldots,x_n)+\lambda{\,}d(x_1,\ldots,x_n)\\
&\geq &(n-2)\max\{\lambda,1-\lambda\}{\,}d(x_1,\ldots,x_n)+d(x_1,\ldots,x_n)\\
&\geq &\frac{n}{2}{\,}d(x_1,\ldots,x_n).
\end{eqnarray*}
\end{itemize}
\end{itemize}
To summarize, we have
$$
K^*_n ~\leq ~ \max\{(n-1)^{-1},2/n\} ~=~ 2/n.
$$
To complete the proof, it suffices to observe that the best constant is attained at any tuple having the stated properties.
\end{proof}

\begin{proof}[Proof of Proposition~\ref{prop:Ksns}]
We first observe that $C_{n,s}\leq 1$. Indeed, using standardness of $d$, we obtain
$$
\sup_{x_1,\ldots,x_n\in X\setminus\{e\}}\frac{d(x_1,\ldots,x_n)}{\sum_{i=1}^nd(x_1,\ldots,x_n)_i^e} ~\leq ~\frac{1}{n-1}~\leq~s.
$$

Now, let $x_1,\ldots,x_n,z\in X$ satisfying $|\{x_1,\ldots,x_n\}|\geq 2$ and set
    $$
    R(x_1,\ldots,x_n;z) ~=~ \frac{d(x_1,\ldots,x_n)}{\sum_{i=1}^nd(x_1,\ldots,x_n)_i^z}
    $$
    and
    $$
    R_s(x_1,\ldots,x_n;z) ~=~ \frac{d_s(x_1,\ldots,x_n)}{\sum_{i=1}^nd_s(x_1,\ldots,x_n)_i^z}{\,}.
    $$
    There are two exclusive cases to consider.
\begin{itemize}
\item Case $e\in\{x_1,\ldots,x_n\}$. We can assume that $x_n=e$. If $z=e$ or $e\in\{x_1,\ldots,x_{n-1}\}$, then
$$
R_s(x_1,\ldots,x_{n-1},e;z) ~=~ R(x_1,\ldots,x_{n-1},e;z) ~\leq ~(n-1)^{-1}{\,}.
$$
If $z\neq e$ and $e\notin\{x_1,\ldots,x_{n-1}\}$, then, using $C_{n,s}\leq 1$, we obtain
\begin{eqnarray*}
\lefteqn{R_s(x_1,\ldots,x_{n-1},e;z)}\\
&\leq & \frac{C_{n,s}{\,}d(x_1,\ldots,x_{n-1},e)}{\sum_{i=1}^{n-1}C_{n,s}{\,} d(x_1,\ldots,x_{n-1},e)_i^z+C_{n,s}{\,}d(x_1,\ldots,x_{n-1},z)}\\
&= & R(x_1,\ldots,x_{n-1},e;z) ~\leq ~(n-1)^{-1}{\,}.
\end{eqnarray*}
\item Case $e\notin\{x_1,\ldots,x_n\}$. If $z\neq e$, then we have
$$
R_s(x_1,\ldots,x_n;z) ~=~ R(x_1,\ldots,x_n;z) ~\leq ~(n-1)^{-1}{\,}.
$$
If $z=e$, then by the definition of $C_{n,s}$, we obtain
$$
R_s(x_1,\ldots,x_n;e)~=~\frac{1}{C_{n,s}}{\,}R(x_1,\ldots,x_n;e)~\leq ~s{\,}.
$$
\end{itemize}
The cases discussed above show that the best constant of $d_s$ is less than or equal to $\max\{(n-1)^{-1},s\}=s$. However, we also have
$$
\sup_{x_1,\ldots,x_n\in X\setminus\{e\}}R_s(x_1,\ldots,x_n;e)~=~ \frac{1}{C_{n,s}}{\,}\sup_{x_1,\ldots,x_n\in X\setminus\{e\}}R(x_1,\ldots,x_n;e) ~=~ s,
$$
which shows that the best constant is exactly $s$.
\end{proof}

\begin{proof}[Proof of Proposition~\ref{prop:asww7df}]
We can assume that $n\geq 3$. Let us first prove the inequality. We proceed by decreasing induction on $k$. The result clearly holds for $k=n$. Suppose that it holds for some integer $k$ satisfying $n+1-1/K^*_n<k\leq n$ and let us prove that it still holds for $k-1$.

Let $x_1,\ldots,x_n,z\in X$. By the induction hypothesis, we have
\begin{equation}\label{eq:65dst1}
d(x_1,\ldots,x_n)~\leq ~\frac{1}{1/K^*_n-n+k}\left(\sum_{i=1}^{k-1}d(x_1,\ldots,x_n)_i^z+d(x_1,\ldots,x_n)_k^z\right)
\end{equation}
and
\begin{equation}\label{eq:65dst2}
d(x_1,\ldots,x_n)_k^z~\leq ~\frac{1}{1/K^*_n-n+k}\left(\sum_{i=1}^{k-1}d(x_1,\ldots,x_n)_{\{k,i\}}^{\{z,x_k\}}+d(x_1,\ldots,x_n)\right),
\end{equation}
where $d(x_1,\ldots,x_n)_{\{k,i\}}^{\{z,x_k\}}$ is the function obtained from $d(x_1,\ldots,x_n)$ by setting its $k$th variable to $z$ and its $i$th variable to $x_k$. We then observe that
\begin{equation}\label{eq:65dst3}
d(x_1,\ldots,x_n)_{\{k,i\}}^{\{z,x_k\}} ~=~ d(x_1,\ldots,x_n)_i^z,\qquad i=1,\ldots,k-1.
\end{equation}
By substituting \eqref{eq:65dst3} into \eqref{eq:65dst2} and then \eqref{eq:65dst2} into \eqref{eq:65dst1}, we obtain
\begin{multline*}
d(x_1,\ldots,x_n) ~\leq ~\frac{1}{1/K^*_n-n+k}\left(1+\frac{1}{1/K^*_n-n+k}\right)\sum_{i=1}^{k-1}d(x_1,\ldots,x_n)_i^z\\
\null +\left(\frac{1}{1/K^*_n-n+k}\right)^2{\,}d(x_1,\ldots,x_n),
\end{multline*}
that is,
$$
d(x_1,\ldots,x_n)~\leq ~\frac{1}{1/K^*_n-n+k-1}{\,}\sum_{i=1}^{k-1}d(x_1,\ldots,x_n)_i^z{\,},
$$
which shows that the result still holds for $k-1$.

Let us now prove the second part of the result. The sufficiency is straightforward. Indeed, we have
$$
d(x_1,\ldots,x_n) ~=~ K^*_n\Big(\sum_{i=1}^kd(x_1,\ldots,x_n)_i^z+(n-k)d(x_1,\ldots,x_n)\Big).
$$
Solving this latter equation for $d(x_1,\ldots,x_n)$ immediately provides \eqref{eq:partkkeq}.

Let us prove the necessity. If \eqref{eq:partkkeq} holds, then in view of \eqref{eq:65dst1}--\eqref{eq:65dst3}, we must have
\begin{eqnarray*}
d(x_1,\ldots,x_n) &=& \frac{1}{1/K^*_n-n+p}\left(\sum_{i=1}^{p-1}d(x_1,\ldots,x_n)_i^z+d(x_1,\ldots,x_n)_p^z\right)\\
d(x_1,\ldots,x_n)_p^z &=& \frac{1}{1/K^*_n-n+p}\left(\sum_{i=1}^{p-1}d(x_1,\ldots,x_n)_i^z+d(x_1,\ldots,x_n)\right)
\end{eqnarray*}
for any $p\in\{k+1,\ldots,n\}$. From these two equations, we derive condition \eqref{eq:partkkeq2}. Moreover, taking $p=n$ in the first equation shows that $K^*_n$ is attained at $(x_1,\ldots,x_n;z)$.
\end{proof}

\begin{proof}[Proof of Proposition~\ref{prop:symm5}]
For any $x_1,\ldots,x_n,z\in X$ and any $S\subseteq\{1,\ldots,n\}$ such that $|S|=k$, we have
$$
d(x_1,\ldots,x_n) ~\leq ~ K^*_{n,k}\sum_{i\in S}d(x_1,\ldots,x_n)_i^z.
$$
Summing this inequality over all possible subsets $S$, we obtain
\begin{eqnarray*}
{n\choose k}{\,}d(x_1,\ldots,x_n) &\leq & K^*_{n,k}\sum_{\textstyle{S\subseteq\{1,\ldots,n\}\atop |S|=k}}\sum_{i\in S}d(x_1,\ldots,x_n)_i^z\\
&\leq & K^*_{n,k}\sum_{i=1}^nd(x_1,\ldots,x_n)_i^z\sum_{\textstyle{S\subseteq\{1,\ldots,n\},{\,}S\ni i\atop |S|=k}}1,
\end{eqnarray*}
where the inner sum reduces to ${n-1\choose k-1}$. This establishes the inequality. The optimality of the constant is obtained by considering the distance defined in Example~\ref{ex:2kfff}, for which we have $K^*_n=2/n$ and $K^*_{n,k}=2/k$.
\end{proof}

\begin{proof}[Proof of Proposition~\ref{prop:Ksnsab}]
Let $x_1,\ldots,x_n,z\in X$ satisfying $|\{x_1,\ldots,x_n\}|\geq 2$ and set
    $$
    R_s(x_1,\ldots,x_x;z) ~=~ \frac{d_s(x_1,\ldots,x_n)}{\sum_{i=1}^nd_s(x_1,\ldots,x_n)_i^z}{\,}.
    $$
    There are two exclusive cases to consider.
\begin{itemize}
\item Case $\{a,b\}\nsubseteq\{x_1,\ldots,x_n\}$. We immediately have $R_s(x_1,\ldots,x_x;z)\leq (n-1)^{-1}$.
\item Case $\{a,b\}\subseteq\{x_1,\ldots,x_n\}$. Let $n_a$ (resp.\ $n_b$) be the exact number of occurrences of $a$ (resp.\ $b$) in $x_1,\ldots,x_n$
\begin{itemize}
\item If $n_a\geq 2$ and $n_b\geq 2$, then
$$
R_s(x_1,\ldots,x_n;z) ~=~ \frac{C_{n,s}}{nC_{n,s}} ~=~ \frac{1}{n}{\,}.
$$
\item If $n_a=1$ and $n_b\geq 2$ (or $n_a\geq 2$ and $n_b=1$), then
$$
R_s(x_1,\ldots,x_n;z) ~=~
\begin{cases}
\frac{C_{n,s}}{nC_{n,s}} ~=~ \frac{1}{n}{\,}, & \text{if $z=a$},\\
\frac{C_{n,s}}{(n-1) C_{n,s}+1} ~=~ \frac{1}{n-1+1/C_{n,s}}{\,}, & \text{if $z\neq a$}.
\end{cases}
$$
\item If $n_a=n_b=1$, then
$$
R_s(x_1,\ldots,x_n;z) ~=~
\begin{cases}
\frac{C_{n,s}}{(n-1) C_{n,s}+1} ~=~ \frac{1}{n-1+1/C_{n,s}}{\,}, & \text{if $z\in\{a,b\}$},\\
\frac{C_{n,s}}{(n-2) C_{n,s}+2} ~=~ \frac{1}{n-2+2/C_{n,s}}{\,}, & \text{if $z\notin\{a,b\}$}.
\end{cases}
$$
\end{itemize}
\end{itemize}
The cases discussed above show that $d_s$ is an $n$-distance and its best constant is
$$
K^*_n ~=~ \frac{1}{n-2+2/C_{n,s}} ~=~ s
$$
and it is attained at any $(a,b,x_3,\ldots,x_n;z)\in X^{n+1}$ such that $x_3,\ldots,x_n,z\in X\setminus\{a,b\}$. Moreover, for such tuples we have
$$
d_s(a,b,x_3,\ldots,x_n) ~=~ C_{n,s} ~=~ \frac{2}{1/s-n+2} = \frac{1}{1/K^*_n-n+2}{\,}\sum_{i=1}^2d_s(a,b,x_3,\ldots,x_n)_i^z{\,},
$$
which shows that $K^*_{n,2}$ exists and $K^*_{n,2}=\frac{1}{1/K^*_n-n+2}$. By Proposition~\ref{prop:asww7df} and Corollary~\ref{cor:asww7dfc}, we also have $K^*_{n,k}=\frac{1}{1/K^*_n-n+k}$ for any $k\in\{2,\ldots,n\}$.
\end{proof}

\begin{proof}[Proof of Lemma~\ref{lemma:f456dfgg}]
By Proposition~\ref{prop:nfknskk}, we have
\begin{multline}\label{mult:8a7d1}
d(x_1,\ldots,x_{k-1},x_k,\ldots,x_k)\\
\leq ~ \frac{1}{k+p-1}{\,}\sum_{i=1}^{k-1}d(x_1,\ldots,x_{k-1},x_k,\ldots,x_k)_i^z \\
\null + \frac{p+1}{k+p-1}{\,}d(x_1,\ldots,x_{k-1},x_k,\ldots,x_k,z).
\end{multline}
Using repetition invariance and Proposition~\ref{prop:nfknskk}, we also have
\begin{multline}\label{mult:8a7d2}
d(x_1,\ldots,x_{k-1},x_k,\ldots,x_k,z) ~=~ d(x_1,\ldots,x_{k-1},x_k,z,\ldots,z)\\
\leq ~ \frac{1}{k-1}{\,}\sum_{i=1}^{k-1}d(x_1,\ldots,x_{k-1},x_k,\ldots,x_k,z)_i^z \\
\null + \frac{1}{k-1}{\,}d(x_1,\ldots,x_{k-1},z,\ldots,z).
\end{multline}
Substituting \eqref{mult:8a7d2} into \eqref{mult:8a7d1}, we obtain the claimed result. To see that the constants are optimal, just take the cardinality based distance with distinct $x_1,\ldots,x_k$ and $z=x_k$.
\end{proof}

\begin{proof}[Proof of Proposition~\ref{prop:azzs54sfd}]
Let $k\in\{2,\ldots,n-1\}$, let $n_1,\ldots,n_k\in\{1,\ldots,n\}$ be such that $n_1+\cdots +n_k=n$, and let $d'\colon X^k\to\R_+$ be the map defined in \eqref{eq:dprime}. Using the symmetry of $d'$, we can rewrite Lemma~\ref{lemma:f456dfgg} with $p=n-k$ as follows. For any $j\in\{1,\ldots,k\}$ and any $x_1,\ldots,x_k,z\in X$, we have
\begin{eqnarray*}
d'(x_1,\ldots,x_k) &\leq & \frac{n}{(k-1)(n-1)}{\,}\sum_{\textstyle{i=1\atop i\neq j}}^k d'(x_1,\ldots,x_k)_i^z \\
&& \null + \frac{n-k+1}{(k-1)(n-1)}{\,}d'(x_1,\ldots,x_k)_j^z{\,}.
\end{eqnarray*}
Summing the latter inequality over $j=1,\ldots,k$, we obtain
\begin{eqnarray*}
\sum_{j=1}^k d'(x_1,\ldots,x_k) &\leq & \frac{n}{(k-1)(n-1)}{\,}\sum_{i=1}^k{\,}\sum_{\textstyle{j=1\atop j\neq i}}^k d'(x_1,\ldots,x_k)_i^z \\
&& \null + \frac{n-k+1}{(k-1)(n-1)}{\,}\sum_{i=1}^k d'(x_1,\ldots,x_k)_i^z,
\end{eqnarray*}
that is,
$$
d'(x_1,\ldots,x_k)~\leq ~ ((k-1)^{-1}+(k(k-1)(n-1))^{-1}){\,}\sum_{i=1}^k d'(x_1,\ldots,x_k)_i^z,
$$
This shows that $d$ satisfies the strong $k$-simplex inequality with the claimed constant $M_{n,k}$, which does not depend on $n_1,\ldots,n_k$.

To see that this constant is optimal and attained, we now construct a standard $n$-distance that is repetition invariant and for which the constant $M_{n,k}$ is attained.

Let $k\in\{2,\ldots,n-1\}$, let $X=\{y_1,\ldots,y_k,e\}$, with $|X|=k+1$, let $d_c$ be the cardinality based $n$-distance, and define $d\colon X^n\to\R_+$ by $d(x_1,\ldots,x_n)=0$ if $|\{x_1,\ldots,x_n\}|=1$, and
$$
d(x_1,\ldots,x_n) ~=~
\begin{cases}
\frac{1}{k-1}{\,}d_c(x_1,\ldots,x_n), & \text{if $e\notin\{x_1,\ldots,x_n\}$},\\
a, & \text{if $e\in\{x_1,\ldots,x_n\}\neq X$},\\
b, & \text{if $e\in\{x_1,\ldots,x_n\}= X$},
\end{cases}
$$
otherwise, where
$$
\frac{(k-1)(n-1)}{k(n-1)+1} ~=~ a ~< ~ b ~=~ \frac{k}{k-1}{\,}a ~=~ \frac{k(n-1)}{k(n-1)+1} ~<~ 1.
$$

Let us show first that $d$ is a standard $n$-distance on $X$. Let $x_1,\ldots,x_n,z\in X$ be such that $|\{x_1,\ldots,x_n\}|\geq 2$. There are three exclusive cases to consider.
\begin{itemize}
\item Case $e\notin\{x_1,\ldots,x_n\}$.
\begin{itemize}
\item If $z\neq e$, then we have
\begin{eqnarray*}
d(x_1,\ldots,x_n) &=& \frac{1}{k-1}{\,}d_c(x_1,\ldots,x_n)\\
&\leq & \frac{1}{n-1}{\,}\sum_{i=1}^n\frac{1}{k-1}{\,}d_c(x_1,\ldots,x_n)_i^z\\
&=& \frac{1}{n-1}{\,}\sum_{i=1}^n d(x_1,\ldots,x_n)_i^z.
\end{eqnarray*}
\item If $z=e$, then assume first that $\{x_1,\ldots,x_n\}\neq X\setminus\{e\}$. We then have
\begin{eqnarray*}
d(x_1,\ldots,x_n) &=& \frac{1}{k-1}{\,}d_c(x_1,\ldots,x_n)~\leq ~\frac{k-2}{k-1}\\
&\leq & \frac{n}{n-1}{\,}a ~=~ \frac{1}{n-1}{\,}\sum_{i=1}^n d(x_1,\ldots,x_n)_i^e.
\end{eqnarray*}
If $\{x_1,\ldots,x_n\}= X\setminus\{e\}=\{y_1,\ldots,y_k\}$, then $d(x_1,\ldots,x_n)=1$ and it is not difficult to see that
$$
\sum_{i=1}^n d(x_1,\ldots,x_n)_i^e ~\geq ~(k-1)a+(n-k+1)b ~=~ n-1.
$$
\end{itemize}
\item Case $e\in\{x_1,\ldots,x_n\}\neq X$. We have
$$
d(x_1,\ldots,x_n) ~=~ a ~\leq ~ \frac{1}{n-1}{\,}\sum_{i=1}^n d(x_1,\ldots,x_n)_i^z
$$
since the sum on the right is greater than $(n-1)a$.
\item Case $e\in\{x_1,\ldots,x_n\}= X$. Let $n_i$ (resp.\ $n_e$) be the number of $y_i$ (resp.\ $e$) among $x_1,\ldots,x_n$.
\begin{itemize}
\item If $z=y_j$ for some $j\in\{1,\ldots,k\}$, then
\begin{eqnarray*}
\lefteqn{\sum_{i=1}^nd(x_1,\ldots,x_n)_i^{y_j} ~=~ \sum_{i=1}^nd(n_1\!\Cdot y_1,\ldots,n_k\!\Cdot y_k,n_e\!\Cdot e)_i^{y_j}}\\
&=&
\begin{cases}
\ell{\,}a+(n-\ell)b, & \text{if $n_e>1$},\\
\ell{\,}a+(n-\ell -1)b+1, & \text{if $n_e=1$},
\end{cases}
\end{eqnarray*}
where $\ell=|\{q\in\{1,\ldots,k\}\setminus\{j\}:n_q=1\}|\leq k-1$. It follows that
\begin{eqnarray*}
d(x_1,\ldots,x_n) ~=~ b ~<~ 1 &=& \frac{(k-1)a+(n-k+1)b}{n-1}\\
&\leq & \frac{1}{n-1}\sum_{i=1}^nd(x_1,\ldots,x_n)_i^z.
\end{eqnarray*}
\item If $z=e$, then
\begin{eqnarray*}
\sum_{i=1}^nd(x_1,\ldots,x_n)_i^e &=& \sum_{i=1}^nd(n_1\!\Cdot y_1,\ldots,n_k\!\Cdot y_k,n_e\!\Cdot e)_i^e\\
&=& \ell a+(n-\ell)b,
\end{eqnarray*}
where $\ell=|\{q\in\{1,\ldots,k\}:n_q=1\}|\leq k$. It follows that
\begin{eqnarray*}
d(x_1,\ldots,x_n) ~=~ b &=& \frac{k{\,}a+(n-k)b}{n-1}\\
&\leq & \frac{1}{n-1}\sum_{i=1}^nd(x_1,\ldots,x_n)_i^z.
\end{eqnarray*}
\end{itemize}
\end{itemize}
In view of the cases discussed above, we see that $d$ is a standard $n$-distance. To see that the claimed constant $M_{n,k}$ is attained, we just observe that
$$
\frac{d'(y_1,\ldots,y_k)}{\sum_{i=1}^k d'(y_1,\ldots,y_k)_i^e} ~=~ \frac{1}{k{\,}a} ~=~ M_{n,k}{\,},
$$
where $d'\colon X^k\to\R_+$ is the (symmetric) map defined in \eqref{eq:dprime}.
\end{proof}

\begin{proof}[Proof of Proposition~\ref{prop:sfda4}]
Let $x_1,\ldots,x_n,x'_1,\ldots,x'_n\in X$ such that
$$
\{x_1,\ldots,x_n\} ~=~ \{x'_1,\ldots,x'_n\}.
$$
We can assume that $|\{x_1,\ldots,x_n\}|\notin\{1,n\}$. Then, there exist pairwise distinct $i,j,k\in\{1,\ldots,n\}$ such that $x_i=x_j\neq x_k$. By nonincreasingness under identification of variables, we can replace $x_i$ or $x_j$ by $x_k$ without changing $d(x_1,\ldots,x_n)$. By iterating this argument, we finally obtain $d(x_1,\ldots,x_n)=d(x'_1,\ldots,x'_n)$.

Let us illustrate this proof. To see that
$$
d(a,b,b,c,c,e) ~=~ d(a,a,b,c,e,e),
$$
we consider the chain of equalities
$$
d(a,b,b,c,c,e) ~=~ d(a,a,b,c,c,e) ~=~ d(a,a,b,c,e,e).\qedhere
$$
\end{proof}

\begin{proof}[Proof of Proposition~\ref{prop:sfda45}]
Let $n_1,\ldots,n_k\in\{1,\ldots,n\}$ such that $n_1+\cdots +n_k=n$. Let also $d'\colon X^k\to\R_+$ be the map defined in \eqref{eq:dprime} and let $x_1,\ldots,x_k,z\in X$. We only need to prove that
$$
d'(x_1,\ldots,x_k) ~\leq ~ K^*_{n,k}{\,}{\,}\sum_{i=1}^kd'(x_1,\ldots,x_k)_i^z{\,}.
$$
Using repetition invariance (see Proposition~\ref{prop:sfda4}) and nonincreasingness under identification of variables, we obtain
$$
d'(x_1,\ldots,x_k) ~\leq ~ d(x_1,\ldots,x_k,z,\ldots,z) ~\leq ~ K^*_{n,k}{\,}\sum_{i=1}^kd(x_1,\ldots,x_k,z,\ldots,z)_i^z
$$
and
$$
d(x_1,\ldots,x_k,z,\ldots,z)_i^z ~=~ d'(x_1,\ldots,x_k)_i^z{\,},\qquad i\in\{1,\ldots,k\}.
$$
This completes the proof.
\end{proof}

\begin{proof}[Proof of Lemma~\ref{lemma:sad7f}]
We proceed by induction. The result trivially holds for $k=1$. Suppose now that the result holds for some $k\in\{1,\ldots,n-1\}$ and let us prove that it still holds for $k+1$. Using Proposition~\ref{prop:nfknskk} and then the induction hypothesis we obtain
\begin{eqnarray*}
d(x_1,\ldots,x_{k+1},z,\ldots,z) &\leq & \frac{1}{k}\sum_{i=1}^{k+1}d(x_1,\ldots,x_{k+1},z,\ldots,z)_i^z\\
&\leq & \frac{1}{k}\sum_{i=1}^{k+1}\sum_{\textstyle{j=1\atop j\neq i}}^{k+1}g(x_j,z)
~=~ \sum_{i=1}^{k+1}g(x_i,z),
\end{eqnarray*}
which completes the proof.
\end{proof}

\section*{Acknowledgments}

The authors would like to thank Bruno Teheux for his insights and useful comments that helped improve this paper. Gergely Kiss is supported by the Hungarian National Foundation for Scientific Research, Grant No. K124749, and the Premium Postdoctoral Fellowship of Hungarian Academy of Sciences.

\end{document}